\documentclass[11pt]{article}
\usepackage{currfile}
\usepackage{amsmath,amsthm,verbatim,amssymb,amsfonts,amscd,graphicx}
\usepackage{graphicx,tikz,caption,subfig}
\usepackage{mathrsfs}
\usepackage{bm}
\usepackage{float}
\usetikzlibrary{calc}
\usepackage[hidelinks]{hyperref}

\usetikzlibrary{patterns}
\usetikzlibrary{arrows, shapes, positioning}
\usetikzlibrary{decorations.markings}
\tikzstyle arrowstyle=[scale=1]
\tikzstyle directed=[postaction={decorate,decoration={markings, mark=at position 0.5 with {\arrow[arrowstyle]{stealth}}}}]
\tikzstyle redirected=[postaction={decorate,decoration={markings, mark=at position 0.5 with {\arrow[arrowstyle]{stealth}}}}]
\usepackage{xcolor}
\usepackage{needspace}
\newif\ifshowchanges
\showchangestrue

\hypersetup{pdftitle={Vizing's theorem for signed multigraphs},
  pdfauthor={You Lu, Jingru Zhao, Li Zhang}}
\usepackage{enumitem}
\setenumerate[1]{itemsep=0pt,partopsep=0pt,parsep=\parskip,topsep=5pt}
\setitemize[1]{itemsep=4pt,partopsep=0pt,parsep=\parskip,topsep=5pt}
\setdescription{itemsep=0pt,partopsep=0pt,parsep=\parskip,topsep=5pt}

\newcommand{\miss}[2]{\overline C_{#1}(#2)}
\newcommand{\col}[3]{#1(h_{#2}^{#3})}

\newtheorem{theorem}{Theorem}[section]
\newtheorem{corollary}[theorem]{Corollary}
\newtheorem{definition}[theorem]{Definition}

\newtheorem{lemma}[theorem]{Lemma}

\newtheorem{Claim}{Claim}

\newcommand{\DM}{{\it Discrete Math.}}
\newcommand{\DAM}{{\it Discrete Appl. Math.}}

\begin{document}
	
\title{Vizing's theorem for signed multigraphs}
	
\author{You Lu\thanks{School of Mathematics and Statistics, Northwestern Polytechnical University,  Xi'an, Shaanxi 710129, China. Email: luyou@nwpu.edu.cn. Partially supported by the National Natural Science Foundation of China (Nos. 12431013, 12671418).}\ \thanks{Corresponding author.}, \ 
~~Jingru Zhao\thanks{School of Mathematics and Statistics, Northwestern Polytechnical University,  Xi'an, Shaanxi 710129, China. Email:  zjr1234562024@163.com.},
~~Li Zhang\thanks{Yan'an University, Yan'an, Shaanxi 716000, China. Email:  lizhang@yau.edu.cn. Partially supported by the National Natural Science Foundation of China (No. 12661075).}
}
	
\date{}

\maketitle
	
\begin{abstract}
We prove that every finite loopless signed multigraph $\Sigma=(G,\sigma)$ satisfies
$\chi'(\Sigma)\leq\Delta(G)+\mu(G)$, where $\chi'(\Sigma)$ is its chromatic index,
and $\Delta(G)$ and $\mu(G)$ are the maximum degree and maximum multiplicity
of $G$, respectively. This bound is sharp, even when both positive and negative edges are present,
and generalizes both Vizing's theorem for ordinary multigraphs and
Behr's theorem for signed simple graphs.

\medskip
\noindent\textbf{Keywords}: signed multigraph, edge coloring, Vizing's theorem.
\end{abstract}

\section{Introduction}
All graphs in this paper are finite and loopless; parallel edges are allowed.
For terminology not defined here, we follow~\cite{BM2008,D2017}.
For a graph $G$, let $d_G(v)$ denote the degree of a vertex $v$, and let
$\Delta(G)$ denote the maximum degree of $G$. For distinct vertices $u,v$, let
$\mu_G(u,v)$ be the number of edges joining $u$ and $v$. The maximum
\emph{multiplicity} of $G$ is
$$
\mu(G)=\max\{\mu_G(u,v):u,v\in V(G),\ u\neq v\}.
$$
A graph is \emph{simple} if $\mu(G)\leq1$.

For integers $a,b$, write $[a,b]=\{j\in\mathbb Z:a\leq j\leq b\}$.
For a nonnegative integer $k$, a \emph{$k$-edge-coloring} of $G$ is a mapping $\phi:E(G)\to[1,k]$
that assigns distinct colors to adjacent edges. The \emph{chromatic index}
$\chi'(G)$ is the least nonnegative integer $k$ for which such a coloring
exists. Clearly, $\chi'(G)\geq\Delta(G)$. The following classical upper
bound was proved independently by Vizing~\cite{Vizing1964} and
Gupta~\cite{Gupta1974}.

\begin{theorem}[Vizing's theorem, \cite{Vizing1964, Gupta1974}]\label{th: vizing}
Every multigraph $G$ satisfies
\[
\chi'(G)\leq\Delta(G)+\mu(G).
\]
\end{theorem}

In particular, every simple graph $G$ satisfies $\chi'(G)\leq\Delta(G)+1$.

A \emph{signed graph} $\Sigma=(G,\sigma)$ consists of a graph $G$ and a
\emph{signature} $\sigma:E(G)\to\{1,-1\}$. An edge $e$ is \emph{positive}
if $\sigma(e)=1$ and \emph{negative} if $\sigma(e)=-1$.
Edge coloring of signed graphs was introduced independently by
Behr~\cite{Behr2020} and Zhang et al.~\cite{ZhangEtAl2020}.
We use Behr's formulation in terms of half-edges. For an edge $e$ incident
with $v$, let $h_e^v$ denote its half-edge at $v$. Write $H(G)$ for the set of all half-edges of $G$, and $H_G(v)$ for the set of half-edges incident with $v$.
For a nonnegative integer $k$, define
$$
M_k=\begin{cases}
[-r,-1]\cup[1,r],&k=2r,\\
[-r,r],&k=2r+1.
\end{cases}
$$
A \emph{$k$-edge-coloring} of $\Sigma$ is a mapping $\gamma:H(G)\to M_k$
such that the colors on $H_G(v)$ are distinct for every $v\in V(G)$ and
$$
\gamma(h_e^u)=-\sigma(e)\gamma(h_e^v)
\qquad(e=uv\in E(G)).
$$
The \emph{chromatic index} $\chi'(\Sigma)$ is the least nonnegative integer
$k$ for which $\Sigma$ has a $k$-edge-coloring. Thus
$\chi'(\Sigma)\geq\Delta(G)$. If all edges are negative, the two half-edges
of each edge receive the same color, so this definition reduces to ordinary
edge coloring and $\chi'(\Sigma)=\chi'(G)$.

Behr~\cite{Behr2020} proved that every signed simple graph
$\Sigma=(G,\sigma)$ satisfies
$$
\chi'(\Sigma)\leq\Delta(G)+1.
$$
Further work on signed simple graphs has addressed
$\Delta$-edge-colorability of signed planar
graphs~\cite{ZhangEtAl2020,ZhangLuZhang2023}, adjacency lemmas for critical
signed graphs~\cite{cao2023,ZhangBroersmaLuZhang2025}, and the dependence of
the chromatic index on the signature~\cite{JanczewskiEtAl2023}.
For signed multigraphs, Steffen and Wolf~\cite{SteffenWolf2023}
established the analogue of Shannon's bound:
$$
\chi'(\Sigma)\leq\left\lfloor\frac{3\Delta(G)}2\right\rfloor.
$$
Our main result extends Vizing's bound to signed multigraphs.

\begin{theorem}\label{th:main}
Every signed multigraph $\Sigma=(G,\sigma)$ satisfies
$$
\chi'(\Sigma)\leq\Delta(G)+\mu(G).
$$
\end{theorem}

Theorem~\ref{th:main} recovers Vizing's theorem when all edges are negative
and Behr's bound when the underlying graph is simple. The bound is sharp even for signed multigraphs containing both positive
and negative edges.  Let $G$ be obtained from $K_5$ by
replacing each edge with $m\geq1$ parallel edges,
and let $\Sigma=(G,\sigma)$ have exactly one positive
edge, with all remaining edges negative. 
In any $k$-edge-coloring of $\Sigma$, the two half-edges
of each negative edge receive the same color, and
negative edges of the same color form a matching.
Since $G$ has five vertices, each color is used on
at most two negative edges. The $10m-1$ negative edges therefore imply that $10m-1\leq2k$. Together with
Theorem~\ref{th:main}, this gives
$$
5m=\left\lceil\frac{10m-1}{2}\right\rceil
\leq\chi'(\Sigma)
\leq\Delta(G)+\mu(G)=5m.
$$

Our proof adapts Behr's signed Kempe chain and fan argument to
multigraphs. Parallel edges allow a neighbor of the fan hinge to occur
more than once. To keep fan shifts proper, we assign distinct missing
colors to the parallel edges joining the hinge to each neighbor.
We first work with an even number of nonzero colors and establish
a local extension criterion. In the remaining parity case, we remove
a suitable matching, color the remaining graph with an even number
of nonzero colors, and assign color $0$ to the edges of the matching.

Section~\ref{se: notation} introduces the recoloring tools and proves
their basic properties. Section~\ref{se:extension} establishes the
extension lemmas, and Section~\ref{se:mainproof} completes the proof
of Theorem~\ref{th:main}.

\section{Notation and recoloring tools}\label{se: notation}

Let $\Sigma=(G,\sigma)$ be a signed multigraph. For $v\in V(G)$, let
$N_G(v)$ and $E_G(v)$ be the sets of neighbors and incident edges of $v$,
respectively. For distinct vertices $u,v$, let $E_G(u,v)$ be the set of
edges joining them, so $\mu_G(u,v)=|E_G(u,v)|$.
Throughout the paper, a \emph{zero-free $k$-edge-coloring} means an edge coloring with colors in $M_k$, where $k$ is a nonnegative even integer.

For $F\subseteq E(G)$, write
$\Sigma-F=(G-F,\sigma|_{E(G)\setminus F})$, and abbreviate
$\Sigma-\{e\}$ to $\Sigma-e$. A coloring of $\Sigma-F$ is also viewed as a
\emph{partial edge coloring} of $\Sigma$, with precisely the edges in $F$
uncolored. For a partial $k$-edge-coloring $\gamma$ with uncolored edge set
$F$, put
$$
C_\gamma(v)=\{\col{\gamma}{e}{v}:e\in E_{G-F}(v)\},
\qquad \miss{\gamma}{v}=M_k\setminus C_\gamma(v).
$$
Colors in $C_\gamma(v)$ are said to be \emph{present} at $v$, and colors in $\miss{\gamma}{v}$ are said to be \emph{missing} at $v$. The \emph{magnitude} of a colored edge $e=uv$ is
$|\col{\gamma}{e}{u}|=|\col{\gamma}{e}{v}|$.
To \emph{color $e=uv$ with color $a$ at $u$} means to assign $a$ to $h_e^u$
and $-\sigma(e)a$ to $h_e^v$. This is possible whenever those colors
are missing at their respective ends.

\subsection{Signed Kempe chains}
For distinct nonzero colors $a,b\in M_k$, let $\pi_{a,b}:M_k\to M_k$ be the permutation that interchanges $a$ with $b$ and $-a$ with $-b$, fixing all other colors. If $b=-a$, it simply interchanges $a$ and $-a$. In both cases,
\begin{equation}\label{eq:pi}
\pi_{a,b}(\pi_{a,b}(z))=z, \qquad
\pi_{a,b}(-z)=-\pi_{a,b}(z) \qquad (z\in M_k).
\end{equation}

\begin{definition}\label{df:kempe}
Let $\gamma$ be a partial $k$-edge-coloring, and suppose that
$a\in\miss{\gamma}{v_0}$ and $b\in C_\gamma(v_0)$ are nonzero.
The \emph{$a/b$-chain} at $v_0$ with respect to $\gamma$ is the maximal trail
$T=v_0e_1v_1\cdots e_mv_m$ satisfying
$$
\gamma(h_{e_1}^{v_0})=b,\qquad
 \col{\gamma}{e_{i+1}}{v_i}
   =\pi_{a,b}\bigl(\col{\gamma}{e_i}{v_i}\bigr)
 \qquad (i\in [1,m-1]).
$$
We denote it by $K_{a,b}^{\gamma}(v_0,v_m)$.
\end{definition}

Along $T$, the half-edges of each positive edge have opposite colors.
When $|a|\neq|b|$, the edge magnitudes alternate between $|b|$ and $|a|$.
Since the half-edge colors at each vertex are distinct, the rule defining
$T$ determines the next edge uniquely whenever it exists.
A \emph{swap} on $T$ replaces the color $z$ of each half-edge of $T$ by $\pi_{a,b}(z)$.

The following formulation of Behr's signed Kempe swap lemma
allows parallel edges and coincident ends. Parallel edges cause no
ambiguity, since each color occurs on at most one half-edge at a vertex.
At every internal occurrence of a vertex, swapping merely interchanges
a pair of colors. Maximality ensures that the new endpoint colors were
missing before the swap; when the ends coincide, these new colors are
distinct because $\pi_{a,b}$ is a permutation. Finally,
Eq.~\eqref{eq:pi} preserves the sign constraint on every edge.

\begin{lemma}[cf.\ \cite{Behr2020}]\label{le:swap}
Let $\gamma$ be a partial $k$-edge-coloring, and let
$T=v_0e_1v_1\cdots e_mv_m$
be an $a/b$-chain. Put $z=\col{\gamma}{e_m}{v_m}$.
Then swapping on $T$ yields a partial $k$-edge-coloring
$\gamma'$ with the same uncolored edges and the following
properties.
\begin{enumerate}
\item
$C_{\gamma'}(v)=C_\gamma(v)$ for every
$v\in V(G)\setminus\{v_0,v_m\}$.

\item
If $v_0\neq v_m$, then
\[
\begin{aligned}
C_{\gamma'}(v_0)
  &=(C_\gamma(v_0)\setminus\{b\})\cup\{a\},\\
C_{\gamma'}(v_m)
  &=(C_\gamma(v_m)\setminus\{z\})\cup\{\pi_{a,b}(z)\}.
\end{aligned}
\]

\item
If $v_0=v_m$, then
\[
C_{\gamma'}(v_0)
  =(C_\gamma(v_0)\setminus\{b,z\})
    \cup\{a,\pi_{a,b}(z)\}.
\]
\end{enumerate}
\end{lemma}

When all edges are negative, an $a/b$-chain uses only the colors
$a$ and $b$ and is a path. For a general signature, a chain may
visit the same vertex more than once. We therefore specify each segment by its position along the chain,
rather than by its end vertices alone.

\subsection{Signed fans}

Let $e_0=uv_0$, and let $\gamma_0$ be a zero-free
$k$-edge-coloring of $\Sigma-e_0$ such that
\begin{equation}\label{eq:assignmentcondition}
|\miss{\gamma_0}{v}|\geq\mu_G(u,v)
\qquad(v\in N_G(u)).
\end{equation}
For each $v\in N_G(u)$, choose pairwise distinct colors
$c(e)\in\miss{\gamma_0}{v}$ for the edges $e\in E_G(u,v)$,
and put
$$
r(e)=-\sigma(e)c(e).
$$
These choices are possible by Eq.~\eqref{eq:assignmentcondition}.
We keep them fixed when defining the fan and its shifts.

\begin{definition}\label{df:fan}
A \emph{signed fan} with hinge $u$ is a maximal sequence
of distinct edges $F=(e_0,e_1,\ldots,e_s)$ incident with
$u$ such that
$$
\col{\gamma_0}{e_i}{u}=r(e_{i-1})
\qquad(i\in[1,s]).
$$
For $i\in[0,s]$, write $e_i=uv_i$, $c_i=c(e_i)$ and
$r_i=r(e_i)$, where the vertices $v_i$ need not be distinct.
\end{definition}

For $t\in[0,s]$, a \emph{shift} of $F$ to $e_t$ assigns
the colors $r_i$ and $c_i$ to the half-edges of $e_i$ at
$u$ and $v_i$, respectively, for $i\in[0,t-1]$.
The edge $e_t$ is uncolored, and all edges outside
$\{e_0,e_1,\ldots,e_t\}$ retain their colors under $\gamma_0$.
We denote the resulting mapping by $\gamma_t$.

\begin{lemma}\label{le:fan-shift}
Let $F=(e_0, e_1,\ldots,e_s)$ be a signed fan. Then the
following hold.
\begin{enumerate}
\item The colors $r_0, r_1,\ldots,r_{s-1}$ are pairwise
distinct. If $r_s\in C_{\gamma_0}(u)$, there is a unique
$\ell\in[0,s-1]$ such that $r_s=r_\ell$.
\item For every $t\in[0,s]$, the mapping $\gamma_t$ is a
zero-free $k$-edge-coloring of $\Sigma-e_t$, and
\begin{equation*}
C_{\gamma_t}(u)=C_{\gamma_0}(u),\qquad
c_i\in\miss{\gamma_t}{v_i}\qquad(i\in[t,s]).
\end{equation*}
\end{enumerate}
\end{lemma}

\begin{proof}
(1) Under $\gamma_0$, the distinct half-edges $h_{e_1}^u,\ldots, h_{e_s}^u$
have colors $r_0,\ldots,r_{s-1}$, and so these colors
are pairwise distinct. If $r_s$ is present at $u$, the maximality of $F$ implies
that $\col{\gamma_0}{e_j}{u}=r_s$ for some $j\in[1,s]$.
Hence $r_s=r_{j-1}$, and the pairwise distinctness of
$r_0,\ldots,r_{s-1}$ gives uniqueness, proving (1).

(2) Fix $t\in[0,s]$. For each $i\in [0, t-1]$, we have $
\col{\gamma_t}{e_i}{u}
=r_i=-\sigma(e_i)c_i
=-\sigma(e_i)\col{\gamma_t}{e_i}{v_i}.
$
All other colored edges retain their original colors.
It remains to verify that the half-edges incident with
each vertex have distinct colors.

At $u$, the colors $r_0,\ldots,r_{t-1}$ originally occur on
$e_1,\ldots,e_t$ and are reassigned to $e_0,\ldots,e_{t-1}$.
Thus their distinctness is preserved and
$C_{\gamma_t}(u)=C_{\gamma_0}(u)$.

At each $v\in N_G(u)$, the newly assigned colors are
pairwise distinct and were missing under $\gamma_0$.
Hence $\gamma_t$ is a zero-free $k$-edge-coloring of
$\Sigma-e_t$.
For $i\in[t,s]$, the assigned colors on the edges joining $u$ to
$v_i$ are pairwise distinct, so no shifted edge at $v_i$ receives $c_i$.
Thus $c_i$ remains missing under $\gamma_t$, proving (2).
\end{proof}

\section{Extension lemmas}\label{se:extension}

The following two lemmas adapt Behr's recoloring argument
for signed simple graphs~\cite{Behr2020} to signed multigraphs.

\begin{lemma}\label{le:1}
Let $\Sigma=(G,\sigma)$ be a signed multigraph with 
$e_0=uv_0\in E(G)$, let $k$ be a positive even integer, and let $\gamma_0$ be a zero-free
$k$-edge-coloring of $\Sigma-e_0$ satisfying
Eq.~\eqref{eq:assignmentcondition}. If there are colors
$a\in\miss{\gamma_0}{u}$ and $d\in\miss{\gamma_0}{v_0}$
with $|a|=|d|$, then $\Sigma$ has a zero-free
$k$-edge-coloring.
\end{lemma}

\begin{proof}
Suppose, to the contrary, that $\Sigma$ has no zero-free
$k$-edge-coloring. The color $-\sigma(e_0)a$ must be
present at $v_0$ under $\gamma_0$, since otherwise
assigning $a$ to $h_{e_0}^u$ and $-\sigma(e_0)a$ to
$h_{e_0}^{v_0}$ would extend $\gamma_0$ to a zero-free $k$-edge-coloring of $\Sigma$.
As $d$ is missing at $v_0$ and $|a|=|d|$, we have
$d=\sigma(e_0)a$.

By Eq.~\eqref{eq:assignmentcondition}, for each $v\in N_G(u)$,
choose pairwise distinct colors $c(e)\in\miss{\gamma_0}{v}$
for the edges $e\in E_G(u,v)$, with
$c(e_0)=\sigma(e_0)a$. Put $r(e)=-\sigma(e)c(e)$, and
let $F=(e_0,e_1,\ldots,e_s)$ be the resulting signed fan.
In particular,
$$
c_0=\sigma(e_0)a,\qquad r_0=-a.
$$
Put $b=r_s$, and shift $F$ to $e_s$, obtaining the
coloring $\gamma_s$. By Lemma~\ref{le:fan-shift},
$$
a\in\miss{\gamma_s}{u},\qquad
c_s=-\sigma(e_s)b\in\miss{\gamma_s}{v_s}.
$$
Under $\gamma_s$, the color $b$ must be present at $u$, and
$-\sigma(e_s)a$ must be present at $v_s$. Indeed, if either were
missing, we could color $e_s$ using the pair $(b,c_s)$
or $(a,-\sigma(e_s)a)$, respectively, contradicting
our assumption.

Consider the signed Kempe chain
$$
T=K_{c_s,-\sigma(e_s)a}^{\gamma_s}(v_s,w).
$$
Swapping $T$ applies $\pi_{a,b}$ and makes
$-\sigma(e_s)a$ missing at $v_s$. This swap must make
$a$ present at $u$, since otherwise $e_s$ could be
colored using $a$ and $-\sigma(e_s)a$.
By Lemma~\ref{le:swap}, we have $w=u$, and the last
half-edge of $T$ has color $b$ under $\gamma_s$.
Lemma~\ref{le:fan-shift}(1) gives a unique
$\ell\in[0,s-1]$ with $r_\ell=b$. Thus $T$ ends
with $v_\ell e_\ell u$.

\begin{Claim}\label{cl:le1-distinct-magnitudes}
$|a|\neq|b|$ and $\ell\in[1,s-1]$.
\end{Claim}

\begin{proof}[Proof of Claim~\ref{cl:le1-distinct-magnitudes}]
Suppose that $|a|=|b|$. Since $a$ is missing at $u$
and $b$ is present there, we have $b=-a$ and $\ell=0$.
The chain $T$ cannot pass through $u$ internally,
as this would require both $a$ and $-a$ at $u$.
Since $T$ ends with $v_0e_0u$, we may write $T=(T-u)e_0u$,
where $T-u$ denotes the initial segment obtained by deleting the last
edge and the final vertex of $T$.

If $E(T-u)=\emptyset$, then $v_s=v_0$ and $\sigma(e_0)a
=\col{\gamma_s}{e_0}{v_0}
=-\sigma(e_s)a$.
Thus $c_s=\sigma(e_s)a=-\sigma(e_0)a$ is missing
at $v_0$ under $\gamma_0$, contradicting the presence
of $-\sigma(e_0)a$ there.

If $E(T-u)\neq\emptyset$, then $T-u$ avoids $u$ and hence
contains no fan edge. Its colors are therefore the same under
$\gamma_0$ and $\gamma_s$. Its first and last half-edge colors are $-c_s$
and $-c_0$, respectively, where
$c_s\in\miss{\gamma_0}{v_s}$ and
$c_0\in\miss{\gamma_0}{v_0}$.
Hence
$$
T-u=K_{c_s,-\sigma(e_s)a}^{\gamma_0}(v_s,v_0).
$$
By Lemma~\ref{le:swap}, swapping $T-u$ in $\gamma_0$
makes $-\sigma(e_0)a$ missing at $v_0$ and leaves
$a$ missing at $u$. Assigning $a$ to $h_{e_0}^u$
and $-\sigma(e_0)a$ to $h_{e_0}^{v_0}$
would give a zero-free $k$-edge-coloring of $\Sigma$,
a contradiction. Therefore $|a|\neq|b|$, which
also implies $\ell\neq0$.
\end{proof}

\begin{Claim}\label{cl:le1-later-edge}
$T$ contains a fan edge $e_j$ with
$j\in[\ell+1,s-1]$.
\end{Claim}

\begin{proof}[Proof of Claim~\ref{cl:le1-later-edge}]
Suppose that every fan edge on $T$ has index at most
$\ell$. Shift $F$ to $e_\ell$, obtaining the coloring
$\gamma_\ell$, and let $R$ be the segment obtained
from $T$ by deleting its last edge $e_\ell$.
The segment $R$ contains at least one edge, since the first edge
of $T$ has magnitude $|a|$, whereas $e_\ell$ has
magnitude $|b|$. Since every fan edge on $R$ has index
less than $\ell$, the colors of $R$ are the same under
$\gamma_\ell$ and $\gamma_s$.

The first and last half-edge colors of $R$ are
$-\sigma(e_s)a$ and $-\sigma(e_\ell)a$, whose
images under $\pi_{a,b}$ are $c_s$ and $c_\ell$.
By Lemma~\ref{le:fan-shift}, the colors $c_s$ and $c_\ell$ are
missing at $v_s$ and $v_\ell$, respectively, under $\gamma_\ell$.
Thus
$$
R=K_{c_s,-\sigma(e_s)a}^{\gamma_\ell}(v_s,v_\ell).
$$
By Lemma~\ref{le:swap}, swapping $R$ in
$\gamma_\ell$ makes $-\sigma(e_\ell)a$ missing
at $v_\ell$ and leaves $a$ missing at $u$. Assigning
$a$ to $h_{e_\ell}^u$ and $-\sigma(e_\ell)a$ to
$h_{e_\ell}^{v_\ell}$ would give a zero-free
$k$-edge-coloring of $\Sigma$, a contradiction.
Hence $T$ contains a fan edge with index greater
than $\ell$. Its index is less than $s$, since
$e_s$ is uncolored.
\end{proof}

Fix an index $j$ as in Claim~\ref{cl:le1-later-edge}.
Since $a$ is missing at $u$, the chain $T$ can use
only the colors $b,-a,-b$ there. Under $\gamma_s$,
the half-edges $h_{e_\ell}^u$ and $h_{e_0}^u$ have
colors $b$ and $-a$, respectively. Since the half-edge
colors at $u$ are distinct, we have
$\col{\gamma_s}{e_j}{u}=-b$.
If $u$ is an internal vertex of $T$, then $T$
must enter and leave $u$ along $e_0$ and $e_j$,
in either order, since $-a$ and $-b$ are the
only colors present at $u$ that are interchanged
by $\pi_{a,b}$. Because $e_j$ lies on $T$ and its last edge is $e_\ell$,
the chain passes through $u$ internally. Moreover, it does so exactly once,
since no edge of a trail is repeated.
Hence $T$ has one of the two forms shown in
Fig.~\ref{fig: structure of T}:
$$
\begin{aligned}
\text{(a)}\quad &T=P e_0u e_j Q e_\ell u,\\
\text{(b)}\quad &T=P e_j u e_0 Q e_\ell u,
\end{aligned}
$$
where $P$ and $Q$ are the indicated segments of $T$; both avoid
$u$ and therefore contain no fan edge.

\begin{figure}[htb]
\scriptsize
\centering
\captionsetup[subfloat]{labelformat=empty,font=scriptsize,justification=centering}
\subfloat[(a) $T=Pe_0ue_jQe_\ell u$]{%
\begin{minipage}[t]{0.48\textwidth}
\centering
\begin{tikzpicture}[scale=0.75]
\path[use as bounding box] (-4.25,-0.55) rectangle (4.55,6.45);

{
\tikzstyle{every node}=[circle,draw,fill=black,
  minimum size=3pt,inner sep=0pt]
\path (0,0) node (u) {};
\path (-3.0,2.70) node (vzero) {};
\path (-2.05,4.00) node (vi) {};
\path (-0.82,4.82) node (vell) {};
\path (0.82,4.82) node (vk) {};
\path (2.05,4.00) node (vj) {};
\path (3.0,2.70) node (vs) {};
}

\draw[dashed,line width=0.6pt] (u)--(vzero) (u)--(vi)
  (u)--(vell) (u)--(vk) (u)--(vj) (u)--(vs);
\draw[line width=0.7pt]
  (1.40,1.51)--(1.60,1.19)
  (1.57,1.663)--(1.77,1.343);

\coordinate (r) at (3.55,4.35);
\coordinate (outer) at (0,6.00);
\coordinate (l) at (-3.55,4.35);
\coordinate (t) at (0,5.42);
\draw[line width=1pt,directed] (vs)--(r)--(outer)--(l)--(vzero);
\draw[line width=1pt] (vzero)--(u);
\draw[line width=1pt,directed] (-2.160,1.944)--(-1.560,1.404);
\draw[line width=1pt] (vk)--(u);
\draw[line width=1pt,directed] (0.328,1.928)--(0.492,2.892);
\draw[line width=1pt,directed] (vk)--(t)--(vell);
\draw[line width=1pt] (vell)--(u);
\draw[line width=1pt,directed] (-0.590,3.470)--(-0.426,2.506);

\node[below] at (u) {$u$};
\node[left] at (vzero) {$v_0$};
\node[above left] at (vell) {$v_\ell$};
\node[above right] at (vk) {$v_j$};
\node[right] at (vs) {$v_s$};
\path (u)--node[pos=0.7,left=0pt] {$e_0$} (vzero);

\path (u)--node[pos=0.7, left=0pt] {$e_\ell$} (vell);
\path (u)--node[pos=0.7,right=0pt] {$e_j$} (vk);

\path (u)--node[pos=0.7,right=0pt] {$e_s$} (vs);

\path (u)--node[pos=0.25,sloped,below=0pt] {$-a$} (vzero);
\node[anchor=east,inner sep=0pt] at (-0.40,1.90) {$b$};
\node[anchor=west,inner sep=0pt] at (0.40,1.90) {$-b$};

\path (vs)--node[pos=0.50,right=-1pt] {$P$} (r);

\path (vk)--node[pos=0.30,above=-1pt] {$Q$} (t);
\end{tikzpicture}
\end{minipage}%
}
\subfloat[(b) $T=Pe_jue_0Qe_\ell u$]{%
\begin{minipage}[t]{0.48\textwidth}
\centering
\begin{tikzpicture}[scale=0.75]
\path[use as bounding box] (-4.25,-0.55) rectangle (4.55,6.45);
{
\tikzstyle{every node}=[circle,draw,fill=black,
  minimum size=3pt,inner sep=0pt]
\path (0,0) node (u) {};
\path (-3.0,2.70) node (vzero) {};
\path (-2.05,4.00) node (vi) {};
\path (-0.82,4.82) node (vell) {};
\path (0.82,4.82) node (vk) {};
\path (2.05,4.00) node (vj) {};
\path (3.0,2.70) node (vs) {};
}

\draw[dashed,line width=0.6pt] (u)--(vzero) (u)--(vi)
  (u)--(vell) (u)--(vk) (u)--(vj) (u)--(vs);
\draw[line width=0.7pt]
  (1.40,1.51)--(1.60,1.19)
  (1.57,1.663)--(1.77,1.343);

\coordinate (r) at (3.55,4.25);
\coordinate (l) at (-3.55,4.25);
\coordinate (t) at (-1.25,5.55);
\draw[line width=1pt,directed] (vs)--(r)--(vk);
\draw[line width=1pt] (vk)--(u);
\draw[line width=1pt,directed] (0.590,3.470)--(0.426,2.506);
\draw[line width=1pt] (vzero)--(u);
\draw[line width=1pt,directed] (-1.200,1.080)--(-1.800,1.620);
\draw[line width=1pt,directed] (vzero)--(l)--(t)--(vell);
\draw[line width=1pt] (vell)--(u);
\draw[line width=1pt,directed] (-0.590,3.470)--(-0.426,2.506);

\node[below] at (u) {$u$};
\node[left] at (vzero) {$v_0$};
\node[below left] at (vell) {$v_\ell$};
\node[above right] at (vk) {$v_j$};
\node[right] at (vs) {$v_s$};
\path (u)--node[pos=0.7,left=0pt] {$e_0$} (vzero);

\path (u)--node[pos=0.7,right=0pt] {$e_\ell$} (vell);
\path (u)--node[pos=0.7,left=0pt] {$e_j$} (vk);

\path (u)--node[pos=0.7,right=0pt] {$e_s$} (vs);

\path (u)--node[pos=0.25,sloped,below=0pt] {$-a$} (vzero);
\node[anchor=east,inner sep=0pt] at (-0.40,1.90) {$b$};
\node[anchor=west,inner sep=0pt] at (0.40,1.90) {$-b$};
\path (vs)--node[pos=0.50,right=-1pt] {$P$} (r);

\path (l)--node[pos=0.50,left=-1pt] {$Q$} (vzero);
\end{tikzpicture}
\end{minipage}%
}

\caption{Two possible forms of the signed Kempe chain $T$. The vertices
$v_0,\ldots,v_s$ are drawn separately for clarity but need not be distinct.}
\label{fig: structure of T}
\end{figure}
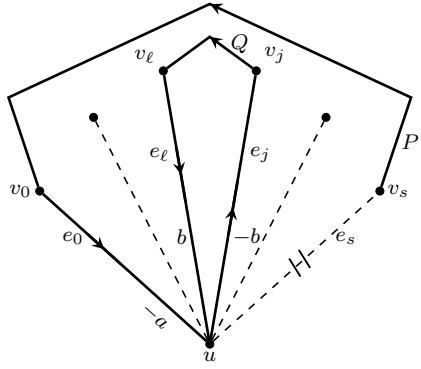
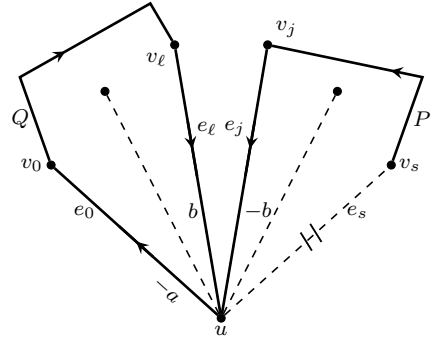

\Needspace{5\baselineskip}
\begin{Claim}\label{cl:le1-chain-order}
$T$ has form~(b).
\end{Claim}

\begin{proof}[Proof of Claim~\ref{cl:le1-chain-order}]
Suppose that $T$ has form~(a). Shift $F$ to
$e_\ell$, obtaining the coloring $\gamma_\ell$.
The segment $Q$ from $v_j$ to $v_\ell$ contains at least one edge,
since otherwise $e_j$ and $e_\ell$ would be
consecutive edges of magnitude $|b|$. As $Q$
contains no fan edge, its colors are the same
under $\gamma_\ell$ and $\gamma_s$.
Its first and last half-edge colors are
$\sigma(e_j)a$ and $-\sigma(e_\ell)a$, whose
images under $\pi_{a,b}$ are $c_j$ and $c_\ell$.
The colors $c_j$ and $c_\ell$ are missing at the corresponding
ends under $\gamma_\ell$, so
$$
Q=K_{c_j,\sigma(e_j)a}^{\gamma_\ell}(v_j,v_\ell).
$$
By Lemma~\ref{le:swap}, swapping $Q$ in
$\gamma_\ell$ makes $-\sigma(e_\ell)a$ missing
at $v_\ell$ and leaves $a$ missing at $u$. Assigning $a$ to $h_{e_\ell}^u$
and $-\sigma(e_\ell)a$ to $h_{e_\ell}^{v_\ell}$
would give a zero-free
$k$-edge-coloring of $\Sigma$, a contradiction.
\end{proof}

By Claim~\ref{cl:le1-chain-order}, we have
$T=P e_j u e_0 Q e_\ell u$. Shift $F$ to $e_j$,
obtaining the coloring $\gamma_j$.
The segment $P$ contains at least one edge, since the first edge
of $T$ has magnitude $|a|$, whereas $e_j$ has
magnitude $|b|$. As $P$ contains no fan edge,
its colors are the same under $\gamma_0$,
$\gamma_j$, and $\gamma_s$. Its first and last
half-edge colors are $-\sigma(e_s)a$ and
$\sigma(e_j)a$, whose images under $\pi_{a,b}$
are $c_s$ and $c_j$. These colors are missing
at the corresponding ends under both $\gamma_0$
and $\gamma_j$. Hence
$$
P=K_{c_s,-\sigma(e_s)a}^{\gamma_0}(v_s,v_j)
 =K_{c_s,-\sigma(e_s)a}^{\gamma_j}(v_s,v_j).
$$

By Lemma~\ref{le:swap}, swapping $P$ in $\gamma_j$
and $\gamma_0$ gives zero-free $k$-edge-colorings
$\gamma_j'$ of $\Sigma-e_j$ and $\eta$ of
$\Sigma-e_0$, respectively. These colorings agree
on all edges outside the fan and satisfy
\begin{equation}\label{eq:exceptionmissing}
a\in\miss{\gamma_j'}{u},\qquad
\sigma(e_j)a\in\miss{\gamma_j'}{v_j},
\end{equation}
and
\begin{equation}\label{eq:etamissing}
a\in\miss{\eta}{u},\qquad
\sigma(e_j)a\in\miss{\eta}{v_j},\qquad
c_0=\sigma(e_0)a\in\miss{\eta}{v_0}.
\end{equation}
The color $c_0$ remains missing at $v_0$ because the
swap changes the sets of present colors only at the
ends of $P$, where the new colors have magnitude
$|b|\neq|a|$.
By Eq.~\eqref{eq:exceptionmissing}, the color $-\sigma(e_j)a$ must be present at $v_j$
under $\gamma_j'$, since otherwise $e_j$ could be
colored using $a$ at $u$ and $-\sigma(e_j)a$ at
$v_j$.

Consider the signed Kempe chain
$$
S=K_{\sigma(e_j)a,-\sigma(e_j)a}^{\gamma_j'}(v_j,w').
$$
By Lemma~\ref{le:swap}, swapping $S$ makes
$-\sigma(e_j)a$ missing at $v_j$. If $a$ were
still missing at $u$ after the swap, assigning
$a$ to $h_{e_j}^u$ and $-\sigma(e_j)a$ to
$h_{e_j}^{v_j}$ would give a zero-free
$k$-edge-coloring of $\Sigma$, a contradiction.
Thus the swap makes $a$ present at $u$.
By Lemma~\ref{le:swap}, this implies that
$w'=u$ and that the last half-edge of $S$
has color $-a$ under $\gamma_j'$.
Since $P$ avoids $u$, we have
$\col{\gamma_j'}{e_0}{u}
=\col{\gamma_j}{e_0}{u}
=r_0=-a$, and so
the last edge of $S$ is $e_0$.
Moreover, $S$ reaches $u$ only at its end, since it uses only the colors
$a$ and $-a$, and $a$ is missing at $u$ under $\gamma_j'$.
Thus $S=(S-u)e_0u$, where $S-u$ denotes the segment preceding
the final edge $e_0$.

If $E(S-u)=\emptyset$, then $v_j=v_0$ and the first
half-edge of $S$ gives $-\sigma(e_j)a=\col{\gamma_j'}{e_0}{v_0}=\sigma(e_0)a$.
Hence $-\sigma(e_0)a=\sigma(e_j)a$ is missing
at $v_0$ under $\eta$ by Eq.~\eqref{eq:etamissing}.
If $E(S-u)\neq\emptyset$, then $S-u$ contains no fan edge,
so its colors are the same under $\eta$ and $\gamma_j'$. Its first and last
half-edge colors are $-\sigma(e_j)a$ and
$-\sigma(e_0)a$. Since $\sigma(e_j)a\in \miss{\eta}{v_j}$ and $\sigma(e_0)a\in \miss{\eta}{v_0}$ by Eq.~\eqref{eq:etamissing}, we have
$$
S-u=K_{\sigma(e_j)a,-\sigma(e_j)a}^{\eta}(v_j,v_0).
$$
By Lemma~\ref{le:swap}, swapping $S-u$ in $\eta$
makes $-\sigma(e_0)a$ missing at $v_0$ and leaves
$a$ missing at $u$.
In either case, assigning $a$ to $h_{e_0}^u$
and $-\sigma(e_0)a$ to $h_{e_0}^{v_0}$ gives
a zero-free $k$-edge-coloring of $\Sigma$.
This contradicts our assumption and completes
the proof.
\end{proof}


\begin{lemma}\label{le:2}
Let $\Sigma=(G,\sigma)$ be a signed multigraph with
$e_0=uv_0\in E(G)$, let $k$ be a positive even integer,
and let $\gamma_0$ be a zero-free $k$-edge-coloring
of $\Sigma-e_0$ satisfying Eq.~\eqref{eq:assignmentcondition} and $\miss{\gamma_0}{u}\neq\emptyset$. Then there exist
an edge $e'=uv'\in E(G)$ and a zero-free
$k$-edge-coloring $\psi$ of $\Sigma-e'$ such that
$u$ and $v'$ have missing colors of the same magnitude.
\end{lemma}

\begin{proof}
Suppose, to the contrary, that no such pair $(e',\psi)$ exists, and fix $a\in\miss{\gamma_0}{u}$.
By Eq.~\eqref{eq:assignmentcondition}, for each
$v\in N_G(u)$, choose pairwise distinct colors
$c(e)\in\miss{\gamma_0}{v}$ for the edges
$e\in E_G(u,v)$. Put $r(e)=-\sigma(e)c(e)$, and
let $F=(e_0,e_1,\ldots,e_s)$ be the resulting signed fan.
Put $b=r_s$, and shift $F$ to $e_s$, obtaining
the coloring $\gamma_s$. By Lemma~\ref{le:fan-shift},
$$
a\in\miss{\gamma_s}{u},\qquad
c_s=-\sigma(e_s)b\in\miss{\gamma_s}{v_s}.
$$
The color $b$ must be present at $u$, and both
$a$ and $-a$ must be present at $v_s$ under
$\gamma_s$. Otherwise, $u$ and $v_s$ would have
missing colors of the same magnitude, so
$(e_s,\gamma_s)$ would satisfy the conclusion, a contradiction.
Since $c_s$ is missing at $v_s$ and
$|c_s|=|b|$, it follows that $|a|\neq|b|$.

Consider the signed Kempe chain
$$
T=K_{c_s,-\sigma(e_s)a}^{\gamma_s}(v_s,w).
$$
Swapping $T$ applies $\pi_{a,b}$ and makes
$-\sigma(e_s)a$ missing at $v_s$. This swap must
make $a$ present at $u$, since otherwise the
resulting coloring, together with $e'=e_s$,
would satisfy the conclusion.
By Lemma~\ref{le:swap}, we have $w=u$, and the
last half-edge of $T$ has color $b$ under
$\gamma_s$. Lemma~\ref{le:fan-shift}(1) gives
a unique $\ell\in[0,s-1]$ with $r_\ell=b$.
Thus $T$ ends with $v_\ell e_\ell u$.

Let $e_j$ be the first fan edge on $T$ when
traversed from $v_s$. Then $j\in[0,s-1]$. Since $a$ is missing at $u$,
we have $\col{\gamma_s}{e_j}{u}\in\{b,-a,-b\}$. If $\col{\gamma_s}{e_j}{u}=-a$, then $r_j=-a$ and $c_j=\sigma(e_j)a$.
Shift $F$ to $e_j$, obtaining the coloring
$\gamma_j$. By Lemma~\ref{le:fan-shift}, $a\in\miss{\gamma_j}{u}$ and $c_j\in\miss{\gamma_j}{v_j}$, so $(e_j,\gamma_j)$ satisfies the conclusion, a contradiction.
Hence $\col{\gamma_s}{e_j}{u}\in\{b,-b\}$.

\begin{Claim}\label{cl:le2-first-edge-direction}
$T$ traverses $e_j$ from $u$ to $v_j$, and
$\col{\gamma_s}{e_j}{u}=-b$.
\end{Claim}

\begin{proof}[Proof of Claim~\ref{cl:le2-first-edge-direction}]
Suppose that $T$ traverses $e_j$ from $v_j$
to $u$, and let $P$ be the segment preceding
$e_j$. The edge $e_j$ has magnitude $|b|$, whereas the first edge of $T$
has magnitude $|a|$. Thus $P$ contains at least one edge.
As $e_j$ is the first fan edge on $T$, the
segment $P$ contains no fan edge.

Shift $F$ to $e_j$, obtaining the coloring
$\gamma_j$. The colors on $P$ are the same
under $\gamma_j$ and $\gamma_s$, and its first
half-edge has color $-\sigma(e_s)a$. Since
$\col{\gamma_s}{e_j}{v_j}=c_j$, its last
half-edge has color $\pi_{a,b}(c_j)$.
Under $\gamma_j$, the colors $c_s$ and $c_j$ are missing at
$v_s$ and $v_j$, respectively. These are precisely the images under
$\pi_{a,b}$ of the first and last half-edge colors of $P$.
Hence
$$
P=K_{c_s,-\sigma(e_s)a}^{\gamma_j}(v_s,v_j).
$$
By Lemma~\ref{le:swap}, swapping $P$ in
$\gamma_j$ gives a zero-free $k$-edge-coloring
$\psi$ of $\Sigma-e_j$ in which
$\pi_{a,b}(c_j)$ is missing at $v_j$.
The color $a$ remains missing at $u$, since $P$ contains no fan edge. Since $|\pi_{a,b}(c_j)|=|a|$, the pair $(e_j,\psi)$
satisfies the conclusion, a contradiction.

Therefore $T$ traverses $e_j$ from $u$ to
$v_j$. The color $b$ at $u$ occurs only on
the last half-edge of $T$, and so
$\col{\gamma_s}{e_j}{u}=-b$.
\end{proof}

By Claim~\ref{cl:le2-first-edge-direction},
the edge preceding $e_j$ along $T$
has color $-a$ at $u$ and is not a fan edge.
At $u$, the chain can use only the three half-edge colors
$b,-a,-b$, each of which occurs at most once. Since it ends with
$v_\ell e_\ell u$, its only fan edges are therefore $e_j$ and
$e_\ell$, with $j\neq\ell$.
Let $Q$ be the segment of $T$ between $e_j$ and $e_\ell$,
from $v_j$ to $v_\ell$. Then $Q$ contains no fan edge.
It contains at least one edge, since otherwise $e_j$ and $e_\ell$
would be consecutive edges of magnitude $|b|$.

Put $t=\min\{j,\ell\}$, and shift $F$ to
$e_t$, obtaining the coloring $\gamma_t$.
The colors on $Q$ are the same under
$\gamma_t$ and $\gamma_s$. Its first and
last half-edge colors are $\sigma(e_j)a$
and $-\sigma(e_\ell)a$, whose images
under $\pi_{a,b}$ are $c_j$ and $c_\ell$,
respectively. By Lemma~\ref{le:fan-shift},
these colors are missing at $v_j$ and
$v_\ell$ under $\gamma_t$. Hence
$$
Q=K_{c_j,\sigma(e_j)a}^{\gamma_t}(v_j,v_\ell).
$$
By Lemma~\ref{le:swap}, swapping $Q$ in
$\gamma_t$ gives a zero-free $k$-edge-coloring
$\psi$ of $\Sigma-e_t$ such that
$$
a\in\miss{\psi}{u},\qquad
\sigma(e_j)a\in\miss{\psi}{v_j},\qquad
-\sigma(e_\ell)a\in\miss{\psi}{v_\ell}.
$$
In either case, $u$ and $v_t$ have missing colors of magnitude
$|a|$. Thus $(e_t,\psi)$ satisfies the conclusion, a contradiction.
\end{proof}

Combining Lemmas~\ref{le:1} and~\ref{le:2}
gives the following corollary.

\begin{corollary}\label{co:extension}
Let $\Sigma=(G,\sigma)$ be a signed multigraph
with $e_0=uv_0\in E(G)$, and let $k$ be a
positive even integer. Suppose that $\Sigma-e_0$
has a zero-free $k$-edge-coloring and that
\begin{equation}\label{eq:localdegree}
k\geq d_G(u),\qquad
k\geq d_G(v)+\mu_G(u,v)\quad(v\in N_G(u)).
\end{equation}
Then $\Sigma$ has a zero-free $k$-edge-coloring.
\end{corollary}

\begin{proof}
Let $\gamma_0$ be a zero-free $k$-edge-coloring
of $\Sigma-e_0$. For every edge $e$ incident
with $u$ and every zero-free $k$-edge-coloring
$\gamma$ of $\Sigma-e$, Eq.~\eqref{eq:localdegree}
gives
$$
|\miss{\gamma}{u}|=k-d_G(u)+1\geq1,
\qquad
|\miss{\gamma}{v}|\geq k-d_G(v)\geq\mu_G(u,v)
\quad(v\in N_G(u)).
$$

Applying Lemma~\ref{le:2} to $\gamma_0$ gives
an edge $e'=uv'$ and a zero-free
$k$-edge-coloring $\psi$ of $\Sigma-e'$
such that $u$ and $v'$ have missing colors
of the same magnitude. The preceding inequalities show that
$\psi$ also satisfies Eq.~\eqref{eq:assignmentcondition}.
Applying Lemma~\ref{le:1} therefore gives a zero-free
$k$-edge-coloring of $\Sigma$.
\end{proof}

\section{Proof of Theorem~\ref{th:main}}\label{se:mainproof}

We first establish two bounds for zero-free edge colorings and then
use them to prove Theorem~\ref{th:main}.
For a multigraph $G$, write
$$
q(G)=2\left\lceil\frac{\Delta(G)+\mu(G)}2\right\rceil,
\qquad
V_\Delta(G)=\{v\in V(G): d_G(v)=\Delta(G)\}.
$$

\Needspace{6\baselineskip}
\begin{lemma}\label{le:evenbound}
Every signed multigraph $\Sigma=(G,\sigma)$ has a zero-free
$q(G)$-edge-coloring.
\end{lemma}

\begin{proof}
We argue by induction on $|E(G)|$. The assertion is immediate when
$E(G)=\emptyset$. Suppose that $E(G)\neq\emptyset$, and choose
an edge $e_0=uv_0$. By induction, $\Sigma-e_0$
has a zero-free $q(G-e_0)$-edge-coloring. Since both $q(G-e_0)$ and $q(G)$ are even and
$q(G-e_0)\leq q(G)$, we have $M_{q(G-e_0)}\subseteq M_{q(G)}$.
Thus this is also a zero-free $q(G)$-edge-coloring. Note that $q(G)\geq d_G(u)$ and for every $v\in N_G(u)$,
$$
q(G)\geq\Delta(G)+\mu(G)\geq d_G(v)+\mu_G(u,v).
$$
By Corollary~\ref{co:extension},
$\Sigma$ has a zero-free $q(G)$-edge-coloring.
\end{proof}

\begin{lemma}\label{le:independent}
Let $\Sigma=(G,\sigma)$ be a signed multigraph with at least
one edge. If $V_\Delta(G)$ is independent and
$k=\Delta(G)+\mu(G)-1$ is even, then $\Sigma$ has a
zero-free $k$-edge-coloring.
\end{lemma}

\begin{proof}
We argue by induction on $|V_\Delta(G)|$. Write
$\Delta=\Delta(G)$ and $\mu=\mu(G)$. Choose $u\in V_\Delta(G)$
and an edge $e_0=uv_0$, and put $H=G-e_0$. Since $V_\Delta(G)$ is
independent, every neighbor of $u$ has degree at most
$\Delta-1$.

If $\Delta(H)+\mu(H)\leq k$, then $q(H)\leq k$ because
$k$ is even. By Lemma~\ref{le:evenbound}, $\Sigma-e_0$ has a
zero-free $k$-edge-coloring. This case
includes the induction base $|V_\Delta(G)|=1$, for which
$\Delta(H)=\Delta-1$.

If $\Delta(H)+\mu(H)>k$, then the inequalities
$\Delta(H)\leq\Delta$ and $\mu(H)\leq\mu$, together with
$k=\Delta+\mu-1$, imply that $\Delta(H)=\Delta$ and $\mu(H)=\mu$.
Moreover, $V_\Delta(H)=V_\Delta(G)\setminus\{u\}$ is independent
and has fewer vertices than $V_\Delta(G)$.
The induction hypothesis therefore gives a zero-free $k$-edge-coloring
of $\Sigma-e_0$.

In either case, $k=\Delta+\mu-1\geq\Delta=d_G(u)$,
and for every $v\in N_G(u)$, $d_G(v)+\mu_G(u,v)\leq\Delta-1+\mu=k$.
Corollary~\ref{co:extension} therefore gives a
zero-free $k$-edge-coloring of $\Sigma$.
\end{proof}

We can now prove the main theorem.

\begin{proof}[Proof of Theorem~\ref{th:main}]
The assertion is immediate when $E(G)=\emptyset$.
We may assume that $G$ has at least one
edge. Write $\Delta=\Delta(G)$ and $\mu=\mu(G)$.
If $\Delta+\mu$ is even, the theorem follows from
Lemma~\ref{le:evenbound}. 
It remains to consider the case where $\Delta+\mu$ is odd.
Set $k=\Delta+\mu-1$, which is even.

Choose a maximal matching $N$ in the
subgraph of $G$ induced by $V_\Delta(G)$, and put $H=G-N$.
Since $\Delta(H)\leq\Delta$ and $\mu(H)\leq\mu$,
either $\Delta(H)+\mu(H)\leq k$ or
$\Delta(H)=\Delta$ and $\mu(H)=\mu$.
In the first case, $q(H)\leq k$ because $k$ is even, so
Lemma~\ref{le:evenbound} gives a zero-free $k$-edge-coloring $\gamma$
of $\Sigma-N$. In the second case, $V_\Delta(H)$ consists precisely
of the vertices of $V_\Delta(G)$ not covered by $N$.
These vertices form an independent set by the maximality of $N$. Lemma~\ref{le:independent} therefore gives a zero-free
$k$-edge-coloring $\gamma$ of $\Sigma-N$.

Recall that $M_{k+1}=M_k\cup\{0\}$. Assign color $0$ to both half-edges of every edge in $N$.
 Since $N$ is a matching and $\gamma$ uses no zero color, the resulting coloring is a $(k+1)$-edge-coloring of $\Sigma$. Therefore,
$\chi'(\Sigma)\leq k+1=\Delta+\mu$.
\end{proof}


\end{document}